\documentclass[11pt]{amsart}

\usepackage{amsmath, amssymb,enumerate}

\newtheorem{theorem}{Theorem}[section]
\newtheorem{lemma}[theorem]{Lemma}
\newtheorem{cor}[theorem]{Corollary}
\newtheorem{prop}[theorem]{Proposition}

\theoremstyle{definition}

\newtheorem{example}[theorem]{Example}

\newtheorem{rem}[theorem]{Remark}

\theoremstyle{remark}
\newcommand{\NN}{\mathbb{N}}
\newcommand{\RR}{\mathbb{R}}
\usepackage{hyperref}
\numberwithin{equation}{section}
\begin{document}

\title[Moment problem for symmetric algebras of lc spaces]{Moment problem for symmetric algebras of locally convex spaces}

\author[M. Ghasemi, M. Infusino, S. Kuhlmann, M. Marshall]{M. Ghasemi$^{*}$, M. Infusino$^{\dagger}$, S. Kuhlmann$^{\dagger}$, \fbox{ M. Marshall$^{*}$}}
\address{$^{*}$ Department of Mathematics and Statistics,\newline \indent
University of Saskatchewan,\newline \indent
Saskatoon, SK. S7N 5E6, Canada}
\email{mehdi.ghasemi@usask.ca}
\address{$^{\dagger}$Fachbereich Mathematik und Statistik,\newline \indent
Universit\"at Konstanz,\newline \indent
78457 Konstanz, Germany.}
\email{maria.infusino@uni-konstanz.de, salma.kuhlmann@uni-konstanz.de}

\begin{abstract} It is explained how a locally convex (lc) topology $\tau$ on a real vector space $V$ extends to a locally multiplicatively convex (lmc) topology $\overline{\tau}$ on the symmetric algebra $S(V)$. This allows the application of the results on lmc topological algebras obtained by Ghasemi, Kuhlmann and Marshall to obtain representations of $\overline{\tau}$-continuous linear functionals $L: S(V)\rightarrow \mathbb{R}$ satisfying $L(\sum S(V)^{2d}) \subseteq [0,\infty)$ (more generally, $L(M) \subseteq [0,\infty)$ for some $2d$-power module $M$ of $S(V)$) as integrals with respect to uniquely determined Radon measures $\mu$ supported by special sorts of closed balls in the dual space of $V$. The result is simultaneously more general and less general than the corresponding result of Berezansky, Kondratiev and \v Sifrin. It is more general because $V$ can be any lc topological space (not just a separable nuclear space), the result holds for arbitrary $2d$-powers (not just squares), and no assumptions of quasi-analyticity are required. It is less general because it is necessary to assume that $L : S(V) \rightarrow \mathbb{R}$ is $\overline{\tau}$-continuous (not just continuous on each homogeneous part of $S(V)$).
\end{abstract}

\maketitle
\vspace{-0.5cm}
\noindent\footnotesize{\it Murray Marshall passed away in May 2015. He worked on this manuscript together with us until the very last days of his life. We lost a collaborator of many years and a wonderful friend. We sorely miss him. (M.~Ghasemi, M.~Infusino, S.~Kuhlmann)}\normalsize

\section{Introduction}
For $n\ge 1$, $\mathbb{R}[\underline{x}]$ denotes the polynomial
ring $\mathbb{R}[\underline{x}]:=\mathbb{R}[x_1,\dots,x_n]$. The
multidimensional moment problem is the following. Given a linear
functional $L : \mathbb{R}[\underline{x}] \rightarrow \mathbb{R}$
and a closed subset $Y$ of $\mathbb{R}^n$ one wants to know when
there exists a nonnegative Radon measure $\mu$ on $\mathbb{R}^n$ supported on
$Y$ such that $L(f) = \int_Y f d\mu$, $\forall f\in
\mathbb{R}[\underline{x}]$.

In this paper, we continue to study this problem in the following
more general set up. Let $A$ be a commutative ring with 1 which is
an $\mathbb{R}$-algebra. $X(A)$ denotes the \it character space \rm
of $A$, i.e., the set of all ring homomorphisms (that send 1 to~1)
$\alpha : A \rightarrow \mathbb{R}$. For any $a \in A$, $\hat{a} : X(A)
\rightarrow \mathbb{R}$ is defined by $\hat{a}(\alpha) = \alpha(a)$.
The only ring homomorphism from $\mathbb{R}$ to itself is the
identity. $X(A)$ is given the coarsest topology such that the
functions $\hat{a}$, $a \in A$, are continuous. For a topological
space~$X$, $C(X)$ denotes the ring of all continuous functions from
$X$ to $\mathbb{R}$. The mapping $a \mapsto \hat{a}$ defines a ring
homomorphism from $A$ into C$(X(A))$. Let $d$ be an integer with $d\ge
1$. By a \it $2d$-power module \rm of $A$ we mean a subset $M$ of
$A$ satisfying $1 \in M, \ M+M \subseteq M \text{ and } a^{2d}M
\subseteq~M \text{ for each } a \in A.$ A \it $2d$-power preordering
\rm of~$A$ is a $2d$-power module of $A$ which is also closed under
multiplication. In the case $d=1$, $2d$-power modules (resp., $2d$-power
preorderings) are referred to as \it quadratic modules \rm (resp.,
\it quadratic preorderings\rm). For a subset $Y$ of $X(A)$,
$\operatorname{Pos}(Y) := \{ a \in A : \hat{a} \ge 0 \text{ on }
Y\}$ is a quadratic preordering of $A$. We denote by $\sum A^{2d}$
the set of all finite sums $\sum a_i^{2d}$, $a_i \in A$. $\sum
A^{2d}$ is the unique smallest $2d$-power module of $A$. $\sum
A^{2d}$ is closed under multiplication, so $\sum A^{2d}$ is also the
unique smallest $2d$-power preordering of $A$. For an arbitrary family $\{g_j\}_{j\in J}$ of elements in $A$ (note that $J$ is an arbitrary index set possibly uncountable), the $2d$-power module of $A$ generated by $\{g_j\}_{j\in J}$ is $M:= \{\sigma_0+\sigma_1g_{j_1}+\ldots+\sigma_sg_{j_s} : s\in\NN, j_1,\ldots,j_s\in J, \sigma_0,\ldots,\sigma_s\in\sum A^{2d}\}.$ For any subset $M$ of  $A$, $X_M:= \{ \alpha \in X(A) :
\hat{a}(\alpha)\ge 0 \ \forall a\in M\}$. If $M= \sum A^{2d}$ then
$X_M = X(A)$. If $M$ is the $2d$-power module of $A$ generated by
$\{g_j\}_{j\in J}$ then $X_M:= \{ \alpha \in X(A) :
\hat{g_j}(\alpha)\ge 0,\, \forall\, j\in J\}.$

A linear functional $L\!:A\!\to\!\mathbb{R}$ is said to be
{\it $M-$positive}, for some $2d-$power module $M$ of $A$, if $L(M)\subseteq
[0,\infty)$. For a linear functional $L:A\to\mathbb{R}$,
one can consider the set of \emph{representing measures}, i.e.,\! the set of all nonnegative Radon measures $\mu$ on $X(A)$ such that
$L(a) = \int_{X(A)} \hat{a} d\mu$ $\forall a\in A$. The {\it moment
problem} in this general setting is to understand the set of
representing measures for a given linear $L : A \rightarrow
\mathbb{R}$. In particular, one wants to know if this set is non-empty
and in case it is non-empty, when it is a singleton set. We note
that the moment problem for $\mathbb{R}[\underline{x}]$ is a special
case. Indeed, ring homomorphisms from $\mathbb{R}[\underline{x}]$ to
$\RR$ correspond to point evaluations $f \mapsto f(\alpha)$, $\alpha
\in \RR^n$ and $X(\mathbb{R}[\underline{x}])$ is identified (as a
topological space) with~$\mathbb{R}^n$.

This general setting includes several interesting instances of the moment problems appearing in applied fields (e.g., statistical mechanics, quantum field theory, spatial statistics, stochastic geometry, etc.), which cannot be reduced to the classical setting because of their intrinsic infinite dimensionality. For more details about those applications see, e.g., \cite{IKR} and \cite{IK}.

Several works have been devoted to the theoretical investigation of moment problems belonging to the general setting introduced above. In \cite{GKM2} and \cite{AJK} the general moment problem for the algebra of polynomials in an arbitrary set of variables
$\{x_i;\> i \in \Omega\}$ is studied. Many recent papers deal with the
general moment problem where the linear functional in question is
continuous for a certain topology. For instance, \cite{MGES} and \cite{L} deal with linear functionals continuous with respect to
weighted norm topologies, generalizing \cite{BCR} and \cite{Schmu78}. In \cite{MGHAS},
\cite{GKM} and \cite{GMW} the authors analyze integral representations of
linear functionals that are continuous with respect to locally
multiplicatively convex topologies. \cite[Theorem 2.1]{BK},
\cite{BS}, \cite{Bor-Yng75}, \cite{Heg75}, \cite[Section 12.5]{Schmu90}, \cite{I}, \cite{IKR} consider linear
functionals on the symmetric algebra of a nuclear space under certain quasi-analyticity assumptions which are less restrictive than continuity. These papers are precursors of the present one, which provides the following criterion to solve the moment problem for linear functionals on the symmetric algebra of a locally convex space.

\begin{theorem}\label{thm-sym} Let $\tau$ be a locally convex topology on an $\mathbb{R}$-vector space $V$ defined by a directed family $\mathcal{S}$ of seminorms on $V$ and let $\bar{\tau}$ be the finest locally multiplicatively
convex topology on the symmetric algebra $S(V)$ extending
$\tau$. Given an integer $d\ge 1$ and a
$2d-$power module $M$ of $S(V)$, a linear functional $L : S(V)\rightarrow
\mathbb{R}$ is $\overline{\tau}$-continuous and $M-$positive if and only if there exists a nonnegative Radon measure defined on the algebraic dual $V^*$ of $V$ representing~$L$ and supported 
by $X_M\cap \overline{B}_i(\rho')$ for some $\rho \in \mathcal{S}$ and some
integer $i\ge 1$.  Here, $\overline{B}_i(\rho')$ denotes the closed ball of radius $i$ centered at the origin in $V^*$ endowed with the operator norm $\rho'$.
\end{theorem}

This is the main result of this paper and will be restated in Section \ref{Sec:lmc} as Corollary \ref{sym}. We derive it from the following identity of cones (see~\cite{GKM})
\begin{equation}\label{id-cones}\overline{M}^{\omega} = \operatorname{Pos}(X_M \cap
\mathfrak{sp}(\omega))
\end{equation}
where $M$ is any $2d-$power module of a locally multiplicatively convex algebra $(A, \omega)$ and $\mathfrak{sp}(\omega)$ is its Gelfand spectrum (i.e., the set of all $\omega$-continuous characters of $A$).
 To this end, the first step is to show how the topology of a real locally convex space $(V, \tau)$ extends to a locally multiplicatively convex topology $\overline{\tau}$ on  the symmetric algebra $S(V)$. The second step is to apply~\eqref{id-cones} for $(A, \omega)=(S(V), \overline{\tau})$. The third and last step consists in getting an explicit description of the Gelfand spectrum $\mathfrak{sp}(\overline{\tau})$ of $S(V)$ w.r.t. $\overline{\tau}$, which in turn provides the support of the representing measure in Theorem \ref{thm-sym}.

For the ease of presentation, we first apply this strategy to the simplest case of locally convex topology, i.e.,\! when it is induced by a single seminorm. Therefore, 
we start by recalling in Section~2 some terminology and notation about seminorms on real vector spaces and results from \cite{GKM} needed to finally get~\eqref{id-cones}, see Theorems \ref{continuity assumption}, \ref{cor jacobi} and \ref{cor ca}). 
In Section~3 we introduce our set up for the graded symmetric algebra
$S(V)$ of a real vector space $V$. Starting with a seminorm $\rho$
on $V$, we here  define its projective extension $\overline{\rho}$ to
$S(V)$ and show that it is submultiplicative, see Proposition~\ref{extension}. We proceed to describe the character space of
$S(V)$ and its Gelfand spectrum with respect to $\overline{\rho}$,
see Proposition \ref{what is sp?}. At the end of that section we are already
in the position to apply the main results of \cite{GKM} to this
situation and get a version of Theorem~\ref{thm-sym} for the case of a seminormed space $(V, \rho)$ (see Corollary \ref{sym0}). It is interesting to point out
that this set up differs from the general case studied in
\cite{GKM2}. Here the free $\mathbb{R}$-algebra $\mathbb{R}[x_i, i\in\Omega]$
is endowed with a topology (i.e., studied as a topological real
algebra) and the linear functionals under consideration are assumed
to be continuous. In Section~4 we continue the exposition of Section 2 but for families of seminorms. Section 5 generalizes the results of Section~3 to the case when $V$ is endowed with a locally convex topology $\tau$ and so contains our main result Theorem \ref{thm-sym} (there Corollary~\ref{sym}). In Section~6 we compare this theorem to the results of \cite{BK}, \cite{BS}, \cite{I} and \cite{IKR}, pointing out the crucial role played by the continuity and the quasi-analyticity assumptions in the localization of the support of the representing measure (see
Remark~\ref{final}). 

\section{Background results}
A $2d-$power module $M$ in $A$ is said to be \it archimedean \rm if
for each $a \in A$ there exists an integer $k$ such that $k \pm a
\in M$. If $M$ is a $2d$-power module of $A$ which is archimedean
then $X_M$ is compact. The converse is false in general (see \cite[\S7.3]{M} or \cite[Example 4.6]{J-P}).

For simplicity, we assume from now on that $A$ is an $\mathbb{R}$-algebra. Let us report here some results from \cite{GKM} about seminormed real algebras which will be used in the forthcoming sections. Recall that a \it seminorm \rm on a $\mathbb{R}$-vector space $V$  is a map $\rho : V \rightarrow [0,\infty)$ such that

\indent
(1) $\forall$ $a\in V$ and $\forall$ $r\in\mathbb{R}$, $\rho(ra)=|r|\rho(a)$, and \newline
\indent
(2) $\forall$ $a,b \in V$, $\rho(a+b)\le\rho(a)+\rho(b)$. \newline
\noindent
A \textit{submultiplicative seminorm} on an $\mathbb{R}$-algebra $A$ is a seminorm $\rho$ on $A$ such that the following holds

\indent
(3) $\forall$ $a,b\in A$, $\rho(ab)\leq\rho(a)\rho(b)$.

\noindent
Let $\rho$ be a submultiplicative seminorm of an $\mathbb{R}$-algebra $A$. Note that if $\rho$ is not identically zero then $\rho(1)\ge 1$.
The \it Gelfand spectrum \rm of $\rho$ is
\begin{align*}
\mathfrak{sp}(\rho) :=& \{ \alpha \in X(A): \alpha \text{ is } \rho\text{-continuous}\}\\ =& \{ \alpha \in X(A) : |\alpha(a)|\le \rho(a) \text{ for all } a\in A\}.
\end{align*}
See \cite[Lemma 3.2]{GKM}.
As explained in \cite[Corollary 3.3]{GKM}, $\mathfrak{sp}(\rho)$ is compact.
\begin{theorem}\label{continuity assumption} If $\rho$ is a submultiplicative seminorm on $A$ and $M$ is a $2d$-power module of $A$ then

(1) $\overline{M}^{\rho} = \operatorname{Pos}(X_M \cap  \mathfrak{sp}(\rho))$ and

(2) $X_{\overline{M}^{\rho}} = X_M\cap \mathfrak{sp}(\rho)$.

\noindent
In particular,  $\overline{\sum A^{2d}}^{\rho} = \operatorname{Pos}(\mathfrak{sp}(\rho))$ and $X_{\overline{\sum A^{2d}}^{\rho}} = \mathfrak{sp}(\rho)$. Here, $\overline{M}^{\rho}$ denotes the closure of $M$ with respect to the seminorm $\rho$.
\end{theorem}

\begin{proof} See \cite[Theorem 3.7]{GKM}.
\end{proof}

Note that (2) in Theorem \ref {continuity assumption} can be deduced from (1) using the Stone-Weierstrass approximation theorem. Consequently, the main result here is (1) and an important element in its proof is the following representation theorem due to T. Jacobi \cite{J}.

\begin{theorem}\label{jacobi} Suppose $M$ is an archimedean $2d$-power module of $A$, $d\ge 1$. Then, for any $a \in A$, the following are equivalent:

(1) $\hat{a} \ge 0$ on $X_M$.

(2) $a+\epsilon \in M$  for all real $\epsilon >0$.
\end{theorem}

The implication (2)$\Rightarrow$(1) is trivial, while  (1)$\Rightarrow$(2) is non-trivial. See \cite{BS-KD}, \cite{K} and \cite{P} for early versions of Jacobi's theorem. See \cite{M} for a short proof in the case $d=1$. See~\cite{GMW} for a short proof in the general case. See~\cite{M0} for a generalization. 

As specified in the introduction, we are interested in the following general version of
the moment problem: Given a linear functional $L : A \rightarrow
\mathbb{R}$, what can be said about the set of nonnegative Radon
measures $\mu$ on $X(A)$ satisfying $L(a) = \int_{X(A)} \hat{a} d\mu$, 
$\forall a\in A$? In particular, one would like to know (i) when this
set is non-empty, and (ii) if it is non-empty, when it is singleton.
Also one wants to understand the support of $\mu$. We say that $\mu$
is supported by some Borel subset $Y$ of $X(A)$ if
$\mu(X(A)\backslash Y)=0$. If $\mu$ is supported by a Borel subset
$Y$ of $X(A)$ then obviously $L$ is
$\operatorname{Pos}(Y)$-positive.
Conversely, if $L$ is $M$-positive for some $2d$-power module $M$ of $A$, does this imply that $\mu$ is supported by $X_M$? One would also like to know for which $\mu$ and for which $p\ge 1$ the natural map $A \rightarrow \mathcal{L}^p(\mu)$, $a\mapsto\hat{a}$, has dense image. Recall if $(X,\mu)$ is a measure space and $f : X \rightarrow \mathbb{R}$ is a 
measurable function, then $$\| f\|_{p,\mu} := \left[\int |f|^p d\mu \right]^{1/p}, \ p\in [1,\infty).$$ The \it Lebesgue space \rm $\mathcal{L}^p(\mu)$, by definition, is the $\mathbb{R}$-vector space  $$\mathcal{L}^p(\mu) := \{ f : X \rightarrow \mathbb{R} : f \text{ is measurable and } \| f\|_{p,\mu} <\infty\}$$ equipped with the norm $\| \cdot \|_{p,\mu}$\footnote{In definition of $\mathcal{L}^p(\mu)$ we assume that each $f$ is a 
representative of the class of all functions $g:X\rightarrow\mathbb{R}$ such that $\|f-g\|_{p,\mu}=0$, otherwise, $\mathcal{L}^p(\mu)$ is merely a seminormed space.}.

Jacobi's result plays a key role for the moment problem in this general setting and indeed it allows to show the following result.
 
 \begin{theorem} \label{cor jacobi} Suppose $M$ is an archimedean $2d-$power module of $A$ and $L : A \rightarrow \mathbb{R}$ a $M-$positive linear functional. Then there exists a nonnegative Radon measure $\mu$ supported by $X_M$ such
that $L(a) = \int_{X_M} \hat{a} d\mu$ $\forall a\in A$.
\end{theorem}

\begin{proof} See  \cite[Corollary 2.6]{GMW}. The conclusion can be also obtained as a consequence of \cite[Theorem 5.5]{MGHAS}.
\end{proof}

Since $M$ is archimedean, $X_M$ is compact, so $\mu$ is the unique nonnegative Radon measure on $X(A)$ satisfying $L(a) = \int_{X(A)} \hat{a} d\mu$ $\forall a\in A$. Also, the image of $A$  in $\mathcal{L}^p(\mu)$ is dense $\forall p \in [1,\infty)$. These are all consequences of the following general result:

\begin{prop} Suppose $\mu$ is a nonnegative Radon measure on $X(A)$ having compact support. Then
\begin{enumerate}[(1)]
\item $\mu$ is determinate, i.e., if $\nu$ is any nonnegative Radon measure on $X(A)$ satisfying $\int_{X(A)} \hat{a} d\nu = \int_{X(A)} \hat{a} d\mu$ $\forall a\in A$, then $\nu = \mu$.
\item The image of $A$  in $\mathcal{L}^p(\mu)$ under the natural map is dense $\forall\,p \in [1,\infty)$.
\end{enumerate}
\end{prop}

\begin{proof} (1) See \cite[Lemma 3.9]{GKM}.
(2) Let $Y$ be a compact subset of $X(A)$ supporting the measure $\mu$. It suffices to show that the step functions $\sum_{j=1}^m r_j\chi_{S_j}$, $r_j \in \mathbb{R}$, $S_j \subseteq Y$ a Borel set, belong to the closure of the image of $A$. Using the triangle inequality we are reduced further to the case $m=1$, $r_1=1$. Let $S \subseteq Y$ be a Borel set. Choose $K$ compact, $U$ open in $Y$ such that $K \subseteq S \subseteq U$, $\mu(U\backslash K)< \epsilon$. By Urysohn's lemma there exists a continuous function $\phi : Y \rightarrow \mathbb{R}$ such that $0\le \phi \le 1$ on $Y$, $\phi = 1$ on $K$, $\phi = 0$ on $Y \backslash U$. Then $\| \chi_S-\phi\|_{p,\mu} \le \epsilon^{1/p}$. Use the Stone-Weierstrass approximation theorem \cite[Theorem 44.7]{W} to get $a\in A$ such that $|\phi(\alpha)-\hat{a}(\alpha)| <\epsilon$ for all $\alpha \in Y$. Then $\|\phi- \hat{a}\|_{p,\mu} \le \epsilon\mu(Y)^{1/p}$. Putting these things together yields $\| \chi_S-\hat{a}\|_{p,\mu} \le \|\chi_S-\phi\|_{p,\mu}+\|\phi - \hat{a}\|_{p,\mu} \le \epsilon^{1/p}+\epsilon\mu(Y)^{1/p}$.
\end{proof}

We conclude this section with a version of Theorem \ref{cor jacobi} which holds for not necessarily archimedean $2d-$power modules of real algebras endowed with a submultiplicative seminorm. This result is actually the dual version of Theorem \ref{continuity assumption} and will be a fundamental to get the main result of this paper.

\begin{theorem}  \label{cor ca}\ \\ For each submultiplicative seminorm $\rho$ on $A$ and each integer $d\ge 1$
there is a natural one-to-one correspondence $L \leftrightarrow \mu$
given by $L(a) = \int _{X(A)}\hat{a} d\mu$ $\forall$ $a\in A$ between
$\rho$-continuous linear functionals $L:~A~\rightarrow~\mathbb{R}$
satisfying $L(\sum A^{2d})\subseteq [0,\infty)$ and nonnegative Radon
measures $\mu$ on $X(A)$ supported by $\mathfrak{sp}(\rho)$. For any
$2d$-power module $M$ of $A$, if $L \leftrightarrow \mu$ under this
correspondence then $L$ is $M$-positive iff $\mu$ is supported by
$X_M\cap \mathfrak{sp}(\rho)$.
\end{theorem}

\begin{proof} See \cite[Corollary 3.8 and Remark 3.10(i)]{GKM}.
\end{proof}

Theorem \ref{cor ca} applies in some interesting cases. See \cite{BCR} and \cite{GMW} for the special case of semigroup algebras. See \cite[Section 4]{GKM} for the application to $*$-seminormed $*$-algebras.
See \cite[Section 5]{GKM} for the application to lmc topological $\mathbb{R}$-algebras. We recall the application to lmc topological $\mathbb{R}$-algebras in a bit more detail in Section 4.

\section{Submultiplicative seminorms on symmetric algebras}

Let $V$ be an $\mathbb{R}$-vector space. We denote by $S(V)$ the symmetric (tensor) algebra of~$V$, i.e., the tensor algebra $T(V)$ factored by the ideal generated by the elements $v\otimes w -w\otimes v$, $v,w\in V$. For notational convenience, given any two elements $f,g\in S(V)$ we denote by $fg$ their symmetric tensor product. If we fix a basis $x_i$, $i\in \Omega$ of $V$, then  $S(V)$ is identified with the polynomial ring $\mathbb{R}[x_i : i\in \Omega]$, i.e., the free $\mathbb{R}$-algebra in commuting variables $x_i$, $i\in \Omega$.
The algebra $S(V)$ is a graded algebra. Denote by $S(V)_k$ the $k$-th homogeneous part of $S(V)$, $k\ge 0$, i.e., the image of $k$-th homogeneous part $V^{\otimes k}$ of $T(V)$ under the canonical map $$\sum_{i=1}^n f_{i1}\otimes \cdots \otimes f_{ik} \mapsto \sum_{i=1}^n  f_{i1}\cdots f_{ik}.$$ Here, $f_{ij} \in V$ for  $i=1,\dots, n$, $j=1,\dots,k$ and $n\ge 1$. Note that $S(V)_0 = \mathbb{R}$ and $S(V)_1 = V$.

Suppose $(V_i,\rho_i)$ are seminormed $\mathbb{R}$-vector spaces, $i=1,2$. The \it projective tensor seminorm \rm $\rho_1 \otimes \rho_2$ on $V_1 \otimes V_2$ is defined by $$(\rho_1 \otimes \rho_2)(f) := \inf\{ \sum_{i=1}^n \rho_1(f_{i1})\rho_2(f_{i2}) : f = \sum_{i=1}^n f_{i1}\otimes f_{i2}, \ f_{ij} \in V_j, \ n \ge 1\}.$$ If $\rho_1$, $\rho_2$ are norms, then $\rho_1 \otimes \rho_2$ is a norm, usually called the \it projective tensor norm \rm or the \it projective cross norm. \rm
See  \cite[Chapter~1, Proposition~1]{Gr} for the proof. 
If $(V_i,\rho_i)$ are seminormed $\mathbb{R}$-vector spaces, $i=1,\dots,k$, then $\rho_1 \otimes \cdots \otimes \rho_k$ is defined recursively, i.e., \begin{align*}(\rho_1 \otimes \cdots &\otimes \rho_k)(f) := \\& \inf\{ \sum_{i=1}^n \rho_1(f_{i1}) \cdots \rho_k(f_{ik}) : f = \sum_{i=1}^n f_{i1}\otimes \cdots \otimes f_{ik}, \ f_{ij} \in V_j, \ n \ge 1\}.\end{align*} If all the $(V_i,\rho_i)$ are equal, say $(V_i,\rho_i) = (V,\rho)$, $i=1,\dots,k$, the associated projective tensor seminorm $\rho_1 \otimes \cdots \otimes \rho_k$ on $V^{\otimes k}$ will be denoted simply by~$\rho^{\otimes k}$.

Suppose now that $\rho$ is a seminorm on $V$ and $\pi_k : V^{\otimes k} \rightarrow S(V)_k$ is the canonical map.
For $k\ge 1$ define $\overline{\rho}_k$ to be the quotient seminorm on $S(V)_k$ induced by $\rho^{\otimes k}$, i.e.,  \begin{align*}\overline{\rho}_k(f) :=& \inf\{ \rho^{\otimes k}(g) : g \in V^{\otimes k}, \ \pi_k(g)=f\} \\ =&  \inf\{ \sum_{i=1}^n \rho(f_{i1})\cdots \rho(f_{ik}) : f = \sum_{i=1}^n f_{i1}\cdots f_{ik}, f_{ij} \in V, n \ge 1 \}.\end{align*}
Define $\overline{\rho}_0$ to be the usual absolute value on $\mathbb{R}$.

\begin{lemma}\label{lemma 1} 
If $k=i+j$, $f\in S(V)_i$, $g\in S(V)_j$, then  $\overline{\rho}_k(fg) \le \overline{\rho}_i(f)\overline{\rho}_j(g)$.
\end{lemma}

\begin{proof} 
Suppose $f\in S(V)_i$, $g\in S(V)_j$ and $f = \sum_p f_{p1}\cdots f_{pi}$, $g = \sum_q g_{q1}\cdots g_{qj}$. Then $fg = \sum_{p,q} f_{p1}\cdots f_{pi}g_{q1}\cdots g_{qj}$, so:

\begin{align*}
\overline{\rho}_k(fg) \le& \sum_{p,q} \rho(f_{p1})\cdots \rho(f_{pi}) \rho(g_{q1})\cdots \rho(g_{qj})\\ =& (\sum_p \rho(f_{p1})\cdots \rho(f_{pi}))(\sum_q \rho(g_{q1})\cdots \rho(g_{qj})).\end{align*} It follows that $\overline{\rho}_k(fg)\le \overline{\rho}_i(f)\overline{\rho}_j(g)$.
\end{proof}

We extend $\rho$ to a seminorm $\overline{\rho}$ on $S(V)$ as follows:
For $f= f_0+\dots + f_{\ell}$, $f_k \in S(V)_k$, $k=0,\dots,\ell$, define 
\begin{equation}\label{proj-ext}
\overline{\rho}(f) := \sum_{k=0}^{\ell} \overline{\rho}_k(f_k).
\end{equation}
We refer to $\overline{\rho}$ as \it the projective extension \rm of $\rho$ to $S(V)$. This extension plays a key role in the following since $\overline{\rho}$ has the fundamental property of being submultiplicative, see Proposition~\ref{extension}. In particular, the $\overline{\rho}-$continuity of a linear functional $L$ on $S(V)$ implies the boundedness of its moments (c.f. Remark \ref{final}-(9)) and so also encodes the boundedness of the support of any representing measure for $L$ (see Corollary \ref{sym0}).

\begin{rem}\label{rem-Fock}
When $(V, \rho)$ is a Hilbert space, i.e., $\rho$ is the norm induced by an inner product $\langle \cdot, \cdot\rangle$, we can clearly run the same algebraic construction of $S(V)=\oplus_{k=0}^\infty S(V)_k$. From a topological point of view, we can endow each $S(V)_k$ with the following inner product 
$$\langle f_1\cdots f_k, g_1\cdots g_k\rangle_k:=\prod_{i=1}^k\langle f_i, g_i\rangle,\quad\forall f_1,\ldots,f_k,g_1,\ldots,g_k\in V$$ 
and consider the completion $\overline{S(V)_k}$ of $(S(V)_k, \langle \cdot, \cdot \rangle_k)$. Denote by $\|\cdot\|_k$ the norm induced by $\langle\cdot, \cdot\rangle_k$ on $S(V)_k$. 

The space $F(V)$ of all $\sum_{k=0}^\infty f_k$ with $f_k\in \overline{S(V)_k}$ s.t. ${\sum_{k=0}^\infty\|f_k\|^2_k}<\infty$ is well-known as the \emph{symmetric (or bosonic) Fock space} over $V$ (see e.g. \cite{F}, \cite[Section II.4]{RS}). 
Then the inner product defined by $$\langle f, g\rangle:=\sum_{k=0}^\infty \langle f_k, g_k\rangle_k,\quad\forall f=\sum_{k=0}^\infty f_k, g=\sum_{k=0}^\infty g_k\in F(V)$$
makes $F(V)$ into a Hilbert space and induces the following norm $$\|f\|:=\sqrt{\sum_{k=0}^\infty \|f_k\|_k^2},\quad\forall f=\sum_{k=0}^\infty f_k\in F(V).$$

Note that $\|\cdot\|$ is an $\ell^2$-type norm constructed from the norms $\|\cdot\|_k$ on $S(V)_k$, while the projective extension $\overline{\rho}$ defined in \eqref{proj-ext} gives in this case an $\ell^1-$type norm constructed from the quotient norms $\overline{\rho}_k$ on $S(V)_k$. Although in this paper we do not consider the completion of $(S(V)_k, \overline{\rho}_k)$, we could have instead worked with it, since the linear functionals here under consideration are all continuous (c.f. Remark \ref{final}-(4)).
\end{rem}

\begin{prop}\label{extension} $\overline{\rho}$ is a submultiplicative seminorm on $S(V)$ extending the seminorm $\rho$ on $V$.
\end{prop}

\begin{proof} Clearly $\overline{\rho}$ is a seminorm on $S(V)$. Also, $\overline{\rho}_1 = \rho$, so $\overline{\rho}$ extends $\rho$. Let $f= \sum_{i=0}^m f_i$, $g= \sum_{j=0}^n g_j$, $f_i\in S(V)_i$, $g_j\in S(V)_j$. Then
\begin{align*}
\overline{\rho}(f g) =& \overline{\rho}(\sum_{i,j} f_i g_j) = \overline{\rho}(\sum_k\sum_{i+j=k} f_i g_j)\\ =& \sum_k \overline{\rho}_k(\sum_{i+j=k} f_i g_j) \le \sum_k \sum_{i+j=k} \overline{\rho}_k(f_i g_j)\\ \le& \sum_k \sum_{i+j=k} \overline{\rho}_i(f_i)\overline{\rho}_j(g_j) = \sum_{i,j} \overline{\rho}_i(f_i)\overline{\rho}_j(g_j)\\ =& (\sum_i\overline{\rho}_i(f_i))(\sum_j \overline{\rho}_j(g_j)) = \overline{\rho}(f)\overline{\rho}(g).
\end{align*}
This proves $\overline{\rho}$ is submultiplicative.
\end{proof}

The algebra $S(V)$ is characterized by the following \it universal property: \rm For each $\mathbb{R}$-linear map $\pi : V \rightarrow A$, where $A$ is an $\mathbb{R}$-algebra (commutative with~1), there exists a unique $\mathbb{R}$-algebra homomorphism $\overline{\pi} : S(V) \rightarrow A$ extending~$\pi$.

Suppose now that $A$ is an $\mathbb{R}$-algebra equipped with submultiplicative seminorm $\sigma$ and $\pi : V \rightarrow A$ is $\mathbb{R}$-linear and continuous with respect to $\rho$ and $\sigma$, i.e., $\exists$ $C>0$ such that $\sigma(\pi(f))\le C\rho(f)$ $\forall$ $f\in V$. Then $\overline{\pi}$ \it need not be continuous \rm with respect to $\overline{\rho}$ and $\sigma$. All one can say in general is

\begin{lemma}\label{l} $\sigma(\overline{\pi}(f)) \le C^k\overline{\rho}_k(f) \ \forall \ f\in S(V)_k.$
\end{lemma}

\begin{proof} Suppose $f = \sum_i f_{i1} \cdots f_{ik}$, $f_{ij} \in V$. Then $\overline{\pi}(f) = \sum_i \pi(f_{i1})\cdots \pi(f_{ik})$ so
$\sigma(\overline{\pi}(f)) \le  \sum_i \sigma(\pi(f_{i1}))\cdots \sigma(\pi(f_{ik})) \le C^k \sum_i \rho(f_{i1})\cdots \rho(f_{ik})$.
This implies  $\sigma(\overline{\pi}(f)) \le C^k\overline{\rho}_k(f)$.
\end{proof}

Of course, if the operator norm of  $\pi$ with respect to $\rho$ and $\sigma$ is $\le 1$ (i.e., if one can choose $C\le 1$) then $\overline{\pi}$ is continuous with respect to $\overline{\rho}$ and $\sigma$.

\begin{prop} \label{functoriality}
If $\pi : (V,\rho) \rightarrow (A,\sigma)$ has operator norm $\le 1$, then the induced algebra homomorphism $\overline{\pi} : (S(V),\overline{\rho}) \rightarrow (A,\sigma)$ has operator norm $\le \sigma(1)$.
\end{prop}

\begin{proof} By Lemma \ref{l} and our assumption concerning the operator norm of $\pi$,   $\sigma(\overline{\pi}(f))\le \overline{\rho}_k(f)$ for all $f \in S(V)_k$, $k\ge 1$. Now let $f=\sum_{k=0}^m f_k$, $f_k\in S(V)_k$, $k=0,\dots,m$. 
Then \begin{align*}
 \sigma(&\overline{\pi}(f)) = \sigma(\sum_{k=0}^m \overline{\pi}(f_k)) \le \sum_{k=0}^m \sigma(\overline{\pi}(f_k)) \le \sigma(\overline{\pi}(f_0))+\sum_{k=1}^m \overline{\rho}_k(f_k) \\ = & \sigma(f_0)+\sum_{k=1}^m \overline{\rho}_k(f_k) \le \sigma(f_0)+\sum_{k=1}^m \sigma(1)\overline{\rho}_k(f_k) = \sigma(1)\sum_{k=0}^m \overline{\rho}_k(f_k) = \sigma(1)\overline{\rho}(f).
\end{align*}
We are assuming here that $\sigma$ is not identically zero (so $\sigma(1)\ge 1$).
If $\sigma$ is identically zero the result is trivial.
\end{proof}

The previous proposition describes a phenomenon similar to the one appearing for the second-quantization of contractive operators on Hilbert spaces. In fact, if $B$ is an operator on a Hilbert space $V$ with operator norm $\leq 1$, then it is possible to construct a second-quantization $\Gamma(B)$ of $B$ by defining $$\Gamma(B)\restriction_{S(V)_k}=\left\{\begin{array}{ll}
 1 & \text{if } k=0\\
 \underbrace{B\otimes \cdots\otimes B}_{n-times}&  \text{if } k\in\NN.
 \end{array}\right.$$
Then the operator $\Gamma(B)$ is densily defined in the Fock space $F(V)$ (see~Remark~\ref{rem-Fock}) and for its operator norm we get
$$\| \Gamma(B)\|=\sup_{n\in\NN_0}\| \Gamma(B)\restriction_{S(V)_k}\|=1,$$
since $\| \Gamma(B)\restriction_{S(V)_k}\|\leq \|B\|^k\leq 1$ for all $k\in\NN$.

The character space $X(S(V))$ of $S(V)$ can be identified with the algebraic dual $V^*= \operatorname{Hom}(V,\mathbb{R})$ of $V$. Indeed, by the universal property of $S(V)$, the mapping $\phi: V^*\to X(S(V))$ defined by $\phi(\ell):=\overline{\ell}, \forall \ell\in V^*$ is an algebraic isomorphism whose inverse is given by $\phi^{-1}(\alpha) = \alpha|_V, \forall \alpha\in X(S(V))$. The topology on $V^*$ is the weak topology, i.e., the coarsest topology such that $v^* \in V^*\mapsto v^*(f)\in \mathbb{R}$ is continuous $\forall f\in V$. If we fix a basis $x_i$, $i\in \Omega$ for~$V$, then $S(V)$ is equal to the polynomial ring $\mathbb{R}[x_i : i \in \Omega]$,  $V^* = \mathbb{R}^{\Omega}$ endowed with the product topology,  and the ring homomorphism $\alpha : S(V) \rightarrow \mathbb{R}$ corresponding to $v^* \in V^*$ is evaluation at $v^*$. We are interested here in the Gelfand spectrum $\mathfrak{sp}(\overline{\rho})$. 
\begin{prop} \label{what is sp?} 
$\mathfrak{sp}(\overline{\rho})$ is naturally identified with the closed unit ball $\overline{B}_1(\rho')$ with respect to the operator norm $\rho'$ on $V^*$ defined by 
$$\rho'(v^*) := \inf\{ C \in [0,\infty) : |v^*(w)| \le C\rho(w) \ \forall w \in V\}.$$
\end{prop}

\begin{proof} Let $\alpha \in \mathfrak{sp}(\overline{\rho})$. Thus $\alpha : S(V)\rightarrow \mathbb{R}$ is an $\mathbb{R}$-algebra homomorphism which is $\overline{\rho}$-continuous, i.e., $|\alpha(f)| \le \overline{\rho}(f)$ $\forall$ $f\in S(V)$. Clearly this implies that $|\alpha|_V(f)| \le \rho(f)$ $\forall$ $f\in V$, so $\rho'(\alpha|_V) \le 1$. Suppose conversely that $v^*\in V^*$, $\rho'(v^*)\le 1$. Denote by $\alpha$ the unique extension of $v^*$ to an $\mathbb{R}$-algebra homomorphism $\alpha : S(V) \rightarrow \mathbb{R}$. Observe that 
$|\alpha(f)| = |v^*(f)| \le \rho(f)$ $\forall$ $f\in V$ so $|\alpha(f)| \le \overline{\rho}(f)$ $\forall$ $f\in S(V)$, by Proposition \ref{functoriality}. Thus $\alpha \in \mathfrak{sp}(\overline{\rho})$.
\end{proof}

\begin{example} 
Let $1\le p < \infty$ and consider the space $V:=\ell_p(\NN)$ of all sequences $x=(x_n)_{n\in\NN}$ such that $x_n\in\RR$ for all $n\in\NN$ and $\sum_{n\in\NN}|x_n|^p<\infty$. We endow $V$ with the classical $\ell_p$-norm, which we denote here just by $\rho$, i.e.,
$$\rho(x):=\left(\sum_{n\in\NN}|x_n|^p\right)^{\frac1p}, \forall x=(x_n)_{n\in\NN}\in\ell_p(\NN).$$
 Then $\mathfrak{sp}(\overline{\rho})$ is naturally identified with $\overline{B}_1(\rho') = [-1,1]^{\NN}$ if $p=1$ and with
$\overline{B}_1(\rho') = \{ y\in\ell_q(\NN) : \sum\limits_{n\in \NN} |y_n|^q \le 1 \}$
if $1< p,q <\infty$ with $\frac{1}{p}+\frac{1}{q}=1$.
\end{example}

\begin{cor}\label{sym0} For each seminormed $\mathbb{R}$-vector space $(V,\rho)$ and each integer $d\ge 1$ there is a one-to-one correspondence $L \leftrightarrow \mu$ given by
$L(f) = \int_{V^*} \hat{f} d\mu$ $\forall$ $f\in S(V)$ between
$\overline{\rho}$-continuous linear functionals $L : S(V)\rightarrow
\mathbb{R}$ satisfying $L(\sum S(V)^{2d})\subseteq [0,\infty)$ and
nonnegative Radon measures $\mu$ on $V^*$ supported by
$\overline{B}_1(\rho')$. For any $2d$-power module $M$ of $S(V)$, if
$L \leftrightarrow \mu$ under this correspondence then $L$
is $M$-positive iff $\mu$ is supported by $X_M\cap
\overline{B}_1(\rho')$. Here,
$$X_M\cap \overline{B}_1(\rho') := \{ v^* \in \overline{B}_1(\rho')
: \hat{g}(\overline{v^*})\ge 0 \ \forall g \in M\}.$$
\end{cor}

\begin{proof} In view of Proposition \ref{what is sp?} this is a direct application of Theorem~\ref{cor ca}.
\end{proof}

As expected, we get the following result\footnote{The result is probably well-known but, as we could not find a reference, we included a proof for the convenience of the reader.}.

\begin{prop} \label{norm case} If $\rho$ is a norm then $\overline{\rho}$ is a norm.
\end{prop}

\begin{proof} It suffices to show $\overline{\rho}_k$ is a norm for each $k\ge 0$. For $k\in \{ 0,1\}$ this is clear. Fix $k\ge 2$ and $f\in S(V)_k$, $f\ne 0$. Fix a basis $x_i$, $i\in \Omega$ for $V$. Thus $f$ is a homogeneous polynomial of degree $k$ in finitely many variables $x_{i_1},\dots, x_{i_n}$, $i_1,\dots,i_n \in \Omega$. Since $f\ne 0$, there exists some non-zero $(a_{i_1},\dots,a_{i_n})\in \mathbb{R}^n$ such that $f(a_{i_1},\dots,a_{i_n})\ne 0$ \cite[Proposition 1.1.1]{M}. Since $f$ is homogeneous of degree $k$, $f(ra_{i_1},\dots,ra_{i_n}) = r^kf(a_{i_1},\dots,a_{i_n}) \ne 0$ for all non-zero $r\in \mathbb{R}$. Let $W\subseteq V$ be the linear span of $x_{i_1},\dots,x_{i_n}$ and let $\phi :W \rightarrow \mathbb{R}$ be the linear functional defined by $\phi(x_{i_j})= a_{i_j}$, $j=1,\dots,n$. 
Since $W$ is finite dimensional and $\rho|_W$ is a norm (because $\rho$ is a norm), any linear functional on $W$ is $\rho|_W$-continuous. In particular,  $\phi$ is $\rho|_W$-continuous so,
by the Hahn-Banach theorem, $\phi$ extends to a $\rho$-continuous linear functional $\Phi : V \rightarrow \mathbb{R}$ having the same operator norm as $\phi$. Scaling, we can assume $\Phi$ and $\phi$ both have operator norm $1$. Thus $\Phi \in \overline{B}_1(\rho')$. Let $\alpha$ be the  element of $\mathfrak{sp}(\overline{\rho})$ corresponding to~$\Phi$. Then $\alpha(f) = f(a_{i_1},\dots,a_{i_n}) \ne 0$. Since $0<|\alpha(f)| \le \overline{\rho}(f)$, this implies $\overline{\rho}_k(f) = \overline{\rho}(f)> 0$.
\end{proof}

\section{Background on LC topologies and LMC topologies}

We begin by recalling some terminology. Let $V$ be an $\mathbb{R}$-vector space. For two seminorms $\rho_1$ and $\rho_2$ on $V$, we write $\rho_1 \succeq \rho_2$ to indicate that there exists $C>0$ such that $C\rho_1(v) \ge \rho_2(v) \ \forall \ v\in V.$ The \it maximum \rm of $\rho_1$ and $\rho_2$ is the seminorm $\rho = \max\{ \rho_1,\rho_2\}$ on $V$ defined by $$\rho(v) := \max\{\rho_1(v), \rho_2(v)\} \ \forall \ v\in V.$$ A \it locally convex \rm (lc) topology on $V$ is just the topology on $V$ generated by some family $\mathcal{S}$ of seminorms on $V$, i.e., it is the coarsest topology on $V$ such that each $\rho \in \mathcal{S}$ is continuous. Closing $\mathcal{S}$ up under the taking of max does not change the topology. In view of this,
there is no harm in assuming, from the beginning, that the family $\mathcal{S}$ is \it directed, \rm i.e., $$\forall \ \rho_1,\rho_2 \in \mathcal{S}, \ \exists \ \rho \in \mathcal{S} \text{ such that } \rho \succeq \max\{ \rho_1,\rho_2\}.$$ 
With this assumption, the open balls $$U_r(\rho) := \{ v \in V : \rho(v)<r\}, \ \rho \in \mathcal{S}, \ r>0$$ form a basis of neighbourhoods of zero (not just a subbasis).

\begin{lemma}\label{lmc2}
Suppose $\tau$ is a locally convex topology on $V$ generated by a directed family $\mathcal{S}$ of seminorms of $V$ and
$L : V \rightarrow \mathbb{R}$ is a $\tau$-continuous linear functional. Then there exists
$\rho \in \mathcal{S}$ such that $L$ is $\rho$-continuous (and conversely, of course).
\end{lemma}

\begin{proof} This is well-known.
The set $\{ v \in V : |L(v)|<1\}$ is an open neighbourhood of the origin in $V$ so there exists $\rho \in \mathcal{S}$ and $r>0$ such that
$U_r(\rho) \subseteq \{ v \in V : |L(v)|<1\}$. Then $U_{r\epsilon}(\rho) = \epsilon U_r(\rho)$ and so we have $L(U_{r\epsilon}(\rho)) = L(\epsilon U_r(\rho)) = \epsilon L(U_r(\rho)) \subseteq \epsilon (-1,1) = (-\epsilon,\epsilon)$ for all $\epsilon>0$, i.e., $L$ is $\rho$-continuous.
\end{proof}

A \it locally multiplicatively convex \rm (lmc) topology $\tau$ on an $\mathbb{R}$-algebra $A$ is just the topology on $A$ generated by some family $\mathcal{S}$ of submultiplicative seminorms on $A$.  See \cite[Section 4.3-2]{BNS}, \cite[Theorem 3.1]{Mal} or \cite[Lemma 1.2]{Mi} for more detail on lmc topologies. 
Again, we can always assume that the family $\mathcal{S}$ is directed.

Let $\tau$ be an lmc topology on an $\mathbb{R}$-algebra $A$. We
denote the Gelfand spectrum of $(A,\tau)$, i.e., the  set of all
$\tau$-continuous $\alpha \in X(A)$, by $\mathfrak{sp}(\tau)$ for
short. Lemma \ref{lmc2} implies that  $$\mathfrak{sp}(\tau) =
\bigcup_{\rho \in \mathcal{S}} \mathfrak{sp}(\rho).$$ Since
$\mathcal{S}$ is directed, this union is directed by inclusion.
Theorem \ref{continuity assumption} extends to general lmc
topologies in an obvious way: If $M$ is any $2d$-power module of $A$
then $$\overline{M}^{\tau} = \operatorname{Pos}(X_M \cap
\mathfrak{sp}(\tau)),$$ where $\overline{M}^{\tau}$ denotes the closure
of $M$ with respect to $\tau$, see
\cite[Theorem 5.4]{GKM}.  In particular, for $M=\sum
A^{2d}$, we get that $$\overline{\sum
A^{2d}}^{\tau} = \operatorname{Pos}(\mathfrak{sp}(\tau)).$$  Theorem \ref{cor ca} also extends to
general lmc topologies. By Lemma \ref{lmc2}, the unique Radon
measure $\mu$ corresponding to a $\tau$-continuous linear functional
$L : A \rightarrow \mathbb{R}$ such that $L$ is $M$-positive is
supported by the compact set $X_M\cap \mathfrak{sp}(\rho)$ for some
$\rho \in \mathcal{S}$. Indeed, for $\rho \in \mathcal{S}$, $\mu$ is
supported by $X_M\cap \mathfrak{sp}(\rho)$ iff $L$ is
$\rho$-continuous.

\section{LMC topologies on symmetric algebras}\label{Sec:lmc}

Let $V$ be an $\mathbb{R}$-vector space. For a seminorm $\rho$ on $V$, we consider the extension of $\rho$ to a submultiplicative seminorm $\overline{\rho}$ of $S(V)$ defined in Section~3.

For seminorms $\rho_1$, $\rho_2$ on $V$, it is important to note that $\rho_1 \succeq \rho_2$ does not imply in general that  $\overline{\rho_1} \succeq \overline{\rho_2}$ but only that $\overline{C\rho_1} \succeq \overline{\rho_2}$ for some $C>0$. This follows from Proposition \ref{functoriality} applied to $\mathfrak{e}:(V, C\rho_1)\hookrightarrow (S(V), \overline{\rho_2})$, where $C >0$ is such that $C\rho_1 \ge \rho_2$ on $V$. Note that Proposition \ref{functoriality} is applicable because the operator norm of $\mathfrak{e}$ is $\leq 1$ by the assumption $\rho_1 \succeq \rho_2$. 

Let $\tau$ be any locally convex topology on $V$. We claim there is a unique finest lmc topology $\overline{\tau}$ on $S(V)$ extending $\tau$. This is pretty clear at this point. Let $\mathcal{S}$ be a family of seminorms on $V$ defining $\tau$. We may assume $\mathcal{S}$ is directed.
Denote by $\overline{\tau}$ the lmc topology on $S(V)$ determined by the directed family of submultiplicative seminorms $\overline{i\rho}$,  $\rho\in \mathcal{S}$, $i\in \{1,2,3,\dots\}$.

\begin{prop} $\overline{\tau}$ extends $\tau$ and is the finest lmc topology on $S(V)$ with this property.
\end{prop}

\begin{proof} The sets $U_r(\rho)$,
$\rho \in \mathcal{S}$, $r>0$ form a basis of neighbourhoods of the origin in $(V,\tau)$, and the sets $U_r(\overline{i\rho})$, 
$\rho \in \mathcal{S}$,  $i\ge 1$,  $r>0$ form a basis of neighbourhoods of the origin in $(S(V),\overline{\tau})$. Since $U_r(\overline{i\rho})\cap V = U_{\frac{r}{i}}(\rho)$ it is clear that $\overline{\tau}$ extends $\tau$. That $\tau$ is the finest lmc topology with this property is a consequence of Proposition \ref{functoriality}: If $N$ is any submultiplicative seminorm on $S(V)$ such that the topology  induced by $N|_V$ is coarser than $\tau$ then there exists $\rho \in \mathcal{S}$ and $i \ge 1$ such that $N(f) \le i\rho(f)$ $\forall$ $f\in V$. By Proposition \ref{functoriality} applied to the inclusion $(V, i\rho)\hookrightarrow (S(V), N)$ having operator norm $\leq 1$, we get that $ N(f) \le N(1)\overline{i\rho}(f)$, $\forall$ $f\in S(V)$. This implies the topology induced by $N$ is coarser than that induced by $\overline{i\rho}$.
\end{proof}

In view of Lemma \ref{lmc2} every $\overline{\tau}$-continuous linear functional $L : S(V) \rightarrow \mathbb{R}$ is $\overline{i\rho}$-continuous for some $\rho \in \mathcal{S}$ and some $i\ge 1$ (and conversely, of course) so Corollary \ref{sym0} can be applied directly to characterize $\overline{\tau}$-continuous linear functionals $L : S(V) \rightarrow \mathbb{R}$ satisfying
$L(\sum S(V)^{2d})\subseteq [0,\infty)$ in terms of measures.
Note also that $(i\rho)' = \frac{1}{i}\rho'$ so $$\mathfrak{sp}(\overline{i\rho}) = \overline{B}_1((i\rho)') = \overline{B}_i(\rho').$$ Consequently, $$\mathfrak{sp}(\overline{\tau}) = \bigcup\limits_{i \ge 1, \rho \in \mathcal{S}} \mathfrak{sp}(\overline{i\rho}) = \bigcup\limits_{i \ge 1, \rho \in \mathcal{S}} \overline{B}_1((i\rho)') =  \bigcup\limits_{i \ge 1, \rho \in \mathcal{S}} \overline{B}_i(\rho').$$

There is special interest in the case where $\tau$ is the finest locally convex topology of $V$, see Corollary \ref{cor-finest} and Proposition \ref{ex-finest} below.

\begin{cor}\label{sym} Let $\tau$ be the locally convex topology on an $\mathbb{R}$-vector space $V$ defined by a directed family $\mathcal{S}$ of seminorms on $V$. For each integer $d\ge 1$ there is a natural one-to-one correspondence $L \leftrightarrow \mu$ given by
$L(f) = \int_{V^*}\hat{f} d\mu$ $\forall$ $f\in S(V)$ between
$\overline{\tau}$-continuous linear functionals $L : S(V)\rightarrow
\mathbb{R}$ satisfying $L(\sum S(V)^{2d})\subseteq [0,\infty)$ and
nonnegative Radon measures $\mu$ on $V^*$ supported by
$\overline{B}_i(\rho')$ for some $\rho \in \mathcal{S}$ and some
integer $i\ge 1$. If $\mu$ is supported by $\overline{B}_i(\rho')$
then $L$ is $\overline{i\rho}$-continuous, and conversely. For any
$2d$-power module $M$ of $S(V)$, if $L \leftrightarrow \mu$ under
this correspondence then $L$ is $M$-positive iff $\mu$ is supported
by $X_M\cap \overline{B}_i(\rho')$. Here,
$$X_M\cap \overline{B}_i(\rho') := \{ v^* \in \overline{B}_i(\rho')
: \hat{g}(\overline{v^*})\ge 0 \ \forall g \in M\}.$$
\end{cor}

\begin{cor}\label{cor-finest}
If $\tau$ is the finest locally convex topology on $V$, then $\overline{\tau}$ is the finest lmc topology on $S(V)$.
\end{cor}
 
Suppose $V$ is finite dimensional with basis $x_1,\dots,x_n$,  so  $S(V)= \mathbb{R}[\underline{x}]$ and $V^* = \mathbb{R}^n$. The finest locally convex topology on $V$ is generated by any fixed norm $\rho$ of $V$. The singleton set $\{\rho\}$ is obviously directed, so $\{ \overline{i\rho} : i \ge 1\}$ generates the finest lmc topology on $S(V)$.

Below we handle the general case.   
\begin{prop} \label{ex-finest}\ \\
Fixed a basis $\{x_i, i\in \Omega\}$ for $V$, for any $r = (r_i)_{i\in \Omega} \in (0,\infty)^{\Omega}$, let $\rho_r$ denote the following norm on $V$
$$\rho_r(\sum_{i\in \Omega} a_ix_i) := \sum_{i \in \Omega} |a_i|r_i.$$  Then the set $\{ \rho_r : r \in (0,\infty)^{\Omega}\}$ generates the finest locally convex topology on $V$ and $\{ \overline{\rho_r} : r \in (0,\infty)^{\Omega}\}$ generates the finest lmc topology on $S(V)$. Furthermore, $\overline{\rho}_r (\sum a_kx^k) = \sum |a_k|r^k$ where $r^k$ denotes the result of evaluating~$x^k$ at $x=r$, i.e., $r^k := \prod_{i\in \Omega} r_i^{k_i}$.
\end{prop}

Note that although the definition of $\rho_r$ depends on the choice of a basis for $V$, the topology generated by the family $\left\{\rho_r: r\in (0,\infty)^\Omega\right\}$, (i.e., the coarsest topology such that all the $\rho_r$'s are continuous) is instead independent of the chosen basis. It is also important to observe that for $V=\ell_1(\NN)$ the norm $\rho_r$ does not coincide with the classical $\ell_1-$norm.
\proof
Since $\{x_i, i\in \Omega\}$ is a basis for $V$, we have that $S(V) = \mathbb{R}[x_i : i \in \Omega]$ and $V^* = \mathbb{R}^{\Omega}$. It is clear that if $\rho$ is any seminorm on $V$ and $\rho(x_i) \le r_i$ $\forall$ $i\in \Omega$ then $\rho \preceq \rho_r$. If $f = \sum_{i\in \Omega} a_ix_i\in V$, then $$\rho(f) \le \sum_{i\in \Omega} \rho(a_ix_i) = \sum_{i\in \Omega} |a_i|\rho(x_i) \le \sum_{i\in \Omega} |a_i|r_i = \rho_r(f).$$ It follows that the set $\{ \rho_r : r \in (0,\infty)^{\Omega}\}$ generates the finest locally convex topology on $V$. Since this set is directed and $j\rho_r = \rho_{jr}$, for any $j\ge 1$, we have that $\{ \overline{\rho_r} : r \in (0,\infty)^{\Omega}\}$ generates the finest lmc topology on $S(V)$. 

Recall that the monomials $x^k := \prod_{i\in \Omega}x_i^{k_i}$ form a basis for $S(V)$ as a vector space over $\mathbb{R}$ and let $\tilde{\rho}_r (\sum a_kx^k) := \sum |a_k|r^k$. Clearly, $\tilde{\rho}_r$ is a submultiplicative seminorm on $S(V)$ and $\tilde{\rho}_r|_V = \rho_r$, so $\tilde{\rho}_r \preceq \overline{\rho}_r$ by Proposition~\ref{functoriality}. On the other hand, the definition of $\overline{\rho}_r$ implies that $\overline{\rho}_r \preceq \tilde{\rho}_r$.
\endproof
Moreover, it is easy to see that $\overline{B}_1(\rho_r') = \prod_{i\in \Omega} [-r_i,r_i].$

\section{Comparison with results in \cite{BK}, \cite{BS}, \cite{I} and \cite{IKR}}
We assume in this section that $(V,\tau)$ belongs to the special class of locally convex spaces considered in \cite[Vol. II, Chapter~5, Section~2]{BK}, \cite{BS}, \cite[Section~3]{I} and \cite{IKR}. Namely, $(V, \tau)$ is assumed to be a separable and nuclear locally convex space, see e.g. \cite[Section 2]{Gr}, \cite[p.100]{Schae}, \cite[Definition~50.1]{Tre67} for the definition.\footnote{Recall that a normed space is not nuclear unless it is finite dimensional. However, every separable infinite dimensional Banach space contains a nuclear subspace (see e.g.~\cite{Val}).} W.l.o.g. we can assume $V$ to be the projective limit of a family $(H_s)_{s\in S}$ of Hilbert spaces ($S$ is an index set containing~$0$) which is directed by topological embedding and such that each $H_s$ is embedded topologically into $H_0$.\footnote{Any separable complete nuclear space is isomorphic to the projective limit of a family of Hilbert spaces (see e.g. \cite[p.103]{Schae}).} Thus $\tau$ is the locally convex topology on $V$ generated by the directed family $\mathcal{S}$ of the norms on $V$ which are induced by the embeddings $V \hookrightarrow H_s$, $s \in S$. In \cite{BK}, \cite{I},  \cite{IKR} the topology $\tau$ is referred to as the \it projective topology \rm on~$V$.

\begin{theorem} \label{nuclear} Let $(V,\tau)$ be a separable nuclear space and let $L : S(V) \rightarrow \mathbb{R}$ be a linear functional. Assume
\begin{enumerate}[(1)]
\item $L(\sum S(V)^2) \subseteq [0,\infty)$;
\item for each $k\ge 0$ the restriction map $L : S(V)_k \rightarrow \mathbb{R}$ is continuous with respect to the locally convex
topology $\overline{\tau}_k$
on $S(V)_k$ induced by the norms $\{\overline{\rho}_k : \rho \in \mathcal{S}\}$; and
\item there exists a countable subset $E$ of $V$ whose linear span is dense in $(V,\tau)$ such that, if
$$ m_0 := \sqrt{L(1)}, \text{ and } m_k := \sqrt{\sup_{f_1, \dots ,f_{2k} \in E} |L(f_1\dots f_{2k})|}, \text{ for } k \ge 1,$$ then the class $C\{ m_k\}$
is quasi-analytic.
\end{enumerate}
Then there exists a Radon measure $\mu$ on the dual space $V^*$ supported by the topological dual $V'$ of $(V,\tau)$ such that $L(f)= \int \hat{f} d\mu$ $\forall$ $f\in S(V)$.
\end{theorem}

\begin{proof} See \cite[Vol. II, Theorem 2.1]{BK} and \cite{BS}.
\end{proof}

\begin{rem} \label{final}\ 
\begin{enumerate}[(1)] 
\item By definition, $\overline{\rho}_k$ is the quotient norm induced by the norm $\rho^{\otimes k}$ on $V^{\otimes k}$ via the surjective linear map $\pi_k : V^{\otimes k} \rightarrow S(V)_k$. It follows that, for any linear map $L : S(V)_k \rightarrow \mathbb{R}$, $L$ is $\overline{\rho}_k$-continuous iff $L\circ \pi_k : V^{\otimes k} \rightarrow \mathbb{R}$ is $\rho^{\otimes k}$-continuous. This is clear.
\item In \cite{BK}, \cite{BS}, \cite{I}, \cite{IKR} the topology on $V^{\otimes k}$ is described in terms of the natural inner products on the $H_s^{\otimes k}$, $s\in S$. The fact that the continuity assumption (2) of Theorem \ref{nuclear} coincides with what is written in \cite[Vol. II, Theorem 2.1]{BK} follows from a well-known consequence of the nuclearity assumption: namely that all cross norms on $V^{\otimes k}$ define the same topology; see \cite[Chapitre 2, Th\'eor\`eme 8]{Gr} or \cite[Theorem 50.1]{Tre67}.
\item Condition (2) of Theorem \ref{nuclear} is equivalent to the assumption in \cite{I} and \cite{IKR} that $m \in \mathcal{F}(V')$. This is nothing but a short way to express the assumption in~\cite{BK} and~\cite{BS} on the starting sequence $m=(m^{(n)})_{n\in\NN_0}$ to be such that each $m^{(n)}\in\left(V^{\otimes n}\right)'$ is a symmetric functional in its $n-$variables. Conditions (2) and (3) combined are equivalent to the so-called \it determining \rm condition in \cite{I} and \cite{IKR} (resp. \it definiteness \rm condition in \cite{BK} and \cite{BS}).
\item $L$ extends by continuity to $L: \bigoplus_{k=0}^{\infty} \overline{S(V)}_k \rightarrow \mathbb{R}$, where here $\overline{S(V)}_k$ denotes the completion of $(S(V)_k, \overline{\tau}_k)$.  As pointed out in \cite{BK}, \cite{BS}, \cite{I}, the Radon measure $\mu$ obtained actually satisfies
$L(f) = \int \hat{f} d\mu$ $\forall$ $f \in \bigoplus_{k=0}^{\infty} \overline{S(V)}_k$. Note that here $ \bigoplus_{k=0}^{\infty} \overline{S(V)}_k$ denotes the algebraic direct sum and not the topological one used for Fock spaces.
\item The proof of Theorem \ref{nuclear} shows that the measure $\mu$ is supported by $H_{-s}$, the Hilbert space dual of $H_s$, for some index $s\in S$ depending on $L$, see \cite[Remark~1, Page 72]{BK}.  Observe that $H_{-s}$ is a countable increasing union of closed balls. Since each such ball is compact in the weak topology, by the Banach-Alaoglu theorem \cite[Theorem 3.15]{Ru}, it follows that $H_{-s}$ is a Borel set in $V^*$.
\item In general it is not known if the measure $\mu$ is unique. If the topological dual~$V'$ of $(V,\tau)$ is a Suslin space then the measure $\mu$ is certainly unique \cite[Theorem 3.6]{I}. However, it is possible to show that $\mu$ is the unique representing measure whose support has the properties described in (5). For more details about the uniqueness of the representing measure in this case see \cite[Section 3]{IKM}. 
\item There is no harm in assuming that the elements of $E$ are chosen to be linearly independent over $\mathbb{R}$, say $E = \{ x_1, x_2, \dots\}$.
Then one can show $$m_k = \sqrt{\sup_{i\ge 1} \{|L(x_i^{2k})|\}},\ \forall\, k\ge 1.$$ This follows from the assumption $L(\sum S(V)^2) \subseteq [0,\infty)$, using the Hurwitz-Reznick theorem. Indeed, by homogenizing \cite[Corollary 3.1.11]{G} in the obvious way, one gets that $$\sum_{i=1}^n \alpha_i x_i^{2k} \pm 2k\prod_{i=1}^n x_i^{\alpha_i}$$ is a sum of squares, so
 $$\left|L(\prod_{i=1}^n x_i^{\alpha_i})\right| \le \max\{ L(x_1^{2k}), \dots, L(x_n^{2k})\}.$$ Here, $\alpha_1, \dots, \alpha_n$ are arbitrary integers satisfying $\alpha_i \ge 0$ and $\sum\limits_{i=1}^n \alpha_i = 2k$.
\item Corollary \ref{sym} is simultaneously more general and less general than Theorem~\ref{nuclear}. It is more general because $(V, \tau)$ can be any locally convex topological space (not just a separable nuclear space) and the result holds for arbitrary $2d$-powers (not just squares).
It is less general because it is necessary to assume that $L : S(V) \rightarrow \mathbb{R}$ is $\overline{\tau}$-continuous.
\item The assumption that $L$ is $\overline{\tau}$-continuous is very strong. It implies not only that the restriction of $L$ to $S(V)_k$ is $\overline{\tau}_k$-continuous, for each $k\ge0$, but also the following strong form of quasi-analyticity: By Lemma \ref{lmc2} there exists $\rho \in \mathcal{S}$ and $C>0$ such that
 $|L(f)|\le C \overline{\rho}(f)$ for all $f\in S(V)$. Then, for any $k\ge 1$ and any $f_1,\dots, f_k \in V$, if $\rho(f_i)\le 1$, $i=1,\dots,k$, then $|L(f_1 \cdots f_k)| \le C \overline{\rho}(f_1 \cdots f_k) \le C \rho(f_1)\cdots \rho(f_k) \le C.$
\item In particular, if $L$ is $\overline{\tau}$-continuous and $(V,\tau)$ is separable, then conditions (2) and (3) of Theorem \ref{nuclear} hold. The fact that condition (2) holds is clear.
Separability implies there exists a countable dense subset $W$ of $V$. Using the previous remark and taking \[
    E:=\{f \in \operatorname{span}_{\mathbb{Q}}W : \rho(f)\leq1\}.
\]
we get the non-empty, countable subset of $V$ fulfilling condition (3). Note that in this case the sequence $\{m_k\}_{k=0}^\infty$ is bounded above by a constant and so the corresponding class $C\{m_k\}$ is quasi-analytic.
\item Corollary~\ref{sym} provides on one hand better information about the support of the representing measure than does Theorem~\ref{nuclear}, on the other hand it covers only to the case of compactly supported measures while this restriction does not appear in Theorem~\ref{nuclear}.
An improvement of both results in this regard is proved in \cite[Theorem 2.3]{IKR} in the special case when $V = \mathcal{C}_c^{\infty}(\mathbb{R}^d)$, the set of all infinitely differentiable functions with compact support contained in $\mathbb{R}^d$. Note that in \cite[Theorem~2.3]{IKR} the functional $L$ is not assumed to be $\bar{\tau}-$continuous and the support of the representing measure is not necessarily compact. Moreover, \cite[Theorem~2.3]{IKR} holds for arbitrary quadratic modules and not just squares as Theorem~\ref{nuclear}.
\end{enumerate}
\end{rem}

\section*{Acknowledgments}
The work of M. Infusino was partially supported by a Marie Curie fellowship of the Istituto Nazionale di Alta Matematica (Grant PCOFUND-GA-2009-245492). The work of M. Marshall was supported by the Natural Sciences and Engineering Research Council of Canada (Grant NSERC: 7854-2013). The authors would like to thank Tobias Kuna and Eli Shamovich for the helpful discussions and the anonymous referee for her or his
valuable suggestions.

\end{document}